\documentclass[psamsfonts]{amsart}

\usepackage{amssymb, amsthm, amsmath, amsfonts, mathrsfs}
\usepackage{enumerate}
\usepackage[all,arc]{xy}
\usepackage{hyperref}
\usepackage{graphicx}
\usepackage{overpic}
\usepackage{xcolor}

\newtheorem{thm}{Theorem}[section]

\newtheorem{prop}[thm]{Proposition}
\newtheorem{lem}[thm]{Lemma}

\theoremstyle{definition}
\newtheorem{defn}[thm]{Definition}

\theoremstyle{remark}
\newtheorem{rem}[thm]{Remark}

\makeatletter
\let\c@equation\c@thm
\makeatother
\numberwithin{equation}{section}

\definecolor{ao}{hsb}{0.67,1,1}
\definecolor{or}{hsb}{0.067,1,1}
\definecolor{green}{hsb}{0.33,1,0.5}

\title{Torus decomposition and foliation detected slopes}

\author{Qingfeng Lyu}
\address{Department of Mathematics\\
  Boston College\\
  Chestnut Hill, MA 02467}
\email{lyuqi@bc.edu}

\begin{document}

\begin{abstract} 
    Let $M_1$ and $M_2$ be knot manifolds and $M=M_1\cup_f M_2$ be the closed manifold obtained by gluing up $M_1$ and $M_2$ via $f:\partial M_1\xrightarrow{\cong} \partial M_2$. We show that if $M$ admits a co-oriented taut foliation, then $f$ identifies some CTF-detected rational boundary slopes of $M_1$ and $M_2$, affirming a conjecture proposed in~\cite{boyer2021slope}.
\end{abstract}

\maketitle

\section{introduction}

The notions of slope detections were introduced by Boyer-Clay~\cite{boyer2017foliations} and Boyer-Gordon-Hu~\cite{boyer2021slope}, providing new viewpoints for the L-space conjecture. In this paper we focus on the foliation detection side:

\begin{defn}[\cite{boyer2021slope}, Definition 5.1]
    Let $M$ be a 3-manifold with $\partial M\cong T^2$. A \textit{boundary slope} [$\alpha$] is an element in the projective space $\mathrm{PH}_1(\partial M,\mathbb{R})$. It is
    \textit{CTF-detected} if there exists a co-oriented taut foliation $\mathcal{F}$ of $M$ intersecting $\partial M$ transversely, such that $\mathcal{F}|_{\partial M}$ is a suspension foliation of the slope [$\alpha$]. It is \textit{strongly CTF-detected} if there exists such an $\mathcal{F}$ where $\mathcal{F}|_{\partial M}$ is in fact a linear foliation of the slope [$\alpha$]. The set of CTF-detected slopes of $M$ is denoted as $\mathcal{D}_{CTF}(M)$.
    \label{defn:ctf}
\end{defn}

A knot manifold is a compact, connected, orientable, and irreducible 3-manifold with boundary an incompressible torus. In~\cite{boyer2021slope} a useful gluing theorem for CTF-detections was proved regarding knot manifolds:

\begin{thm}[\cite{boyer2021slope}, Theorem 5.2]
    Suppose that $M_1$ and $M_2$ are knot manifolds and $M=M_1\cup_fM_2$, where $f:\partial M_1\rightarrow \partial M_2$ is the gluing homeomorphism. If $f$ identifies CTF-detected rational slopes $[\alpha_1]\in \mathcal{D}_{CTF}(M_1)$ and $[\alpha_2]\in \mathcal{D}_{CTF}(M_2)$, then $M$ admits a co-orientable taut foliation.
    \label{thm:gluing}
\end{thm}

The authors conjectured that the converse to this theorem should also be true (see \cite{boyer2021slope}, Conjecture 1.2). In this paper we affirm this conjecture. In fact, we show that:

\begin{thm}
    Suppose $M=M_1\cup_f M_2$ where $M_1,M_2$ are knot manifolds and $f:\partial M_1\xrightarrow{\cong} \partial M_2$. Then $M$ admits a co-oriented taut foliation if and only if $f$ identifies a rational CTF-detected slope on $\partial M_1$ with a rational CTF-detected slope on $\partial M_2$.
    \label{thm:main}
\end{thm}

The outline of the proof is as follows. Suppose $M=M_1\cup_T M_2$ (where $T=\partial M_1=\partial M_2$) admits a co-oriented taut foliation. We will first show that there is some slope of $T$ detected in both $M_1$ and $M_2$, and then show that such a slope can always be chosen to be rational. In the first step we consider two cases separately: when $M$ is a rational homology sphere and when $M$ has positive first Betti number. In the former case we use torus decomposition, while in the latter case we apply Gabai's sutured manifold decompositions.

In the second step dealing with irrational slopes, we will actually prove that any irrational detected slope lies in the interior of the set of detected boundary slopes of a knot manifold (Proposition~\ref{prop:knotirr}). These results follow almost immediately from Li's work on laminar branched surfaces~\cite{li2002laminar}\cite{li2003boundary}, and might be of independent interest. More generally, we can obtain a generalisation to~\cite{boyer2017foliations}, Proposition 6.10, which deals with irrational slopes of graph manifolds.

The authors of~\cite{boyer2021slope} also proved a multislope version of the gluing theorem (see \cite{boyer2021slope}, Theorem 7.6) and conjectured its converse. Slightly modifying our proof to do the above two steps simultaneously, we also provide an affirmative answer to this (multislope version) conjecture for CTF-detections.

\begin{rem}
    In the proof of~\cite{boyer2021slope}, Theorem 7.6, the authors actually proved a relative version of Theorem~\ref{thm:gluing}. That is, $M_1$ and $M_2$ can have multiple boundary tori; moreover, modifying the foliations in $M_1$ and $M_2$ and gluing them together along some boundary torus won't affect the slopes detected by the foliations on other boundary tori. We will essentially use this fact in this paper.
    ~\label{rem:multigluing}
\end{rem}

\hspace*{\fill}

\textbf{Organisation of the paper} In section~\ref{sec:2} we find a slope of $T$ detected in both $M_1,M_2$ for $M$ a rational homology sphere. In section~\ref{sec:3} we do this for $M$ with positive first Betti number. In section~\ref{sec:4} we briefly review the laminar branched surface theory and deal with irrational slopes. In section~\ref{sec:5} we discuss the converse in multislope settings.

\hspace{6pt}

\textbf{Acknowledgements} The author thanks his advisor Tao Li for many helpful conversations, and especially for pointing out the relation between irrational detected slopes and laminar branched surfaces. We thank Yaoping Xie for helpful discussions on some technical details in the paper. We thank the referee for their careful reading of the paper, and many valuable suggestions and corrections that improved the overall exposition.

\section{Torus decomposition and rational homology spheres}
\label{sec:2}

In this section we prove the converse to Theorem~\ref{thm:gluing} when $M$ is a rational homology sphere. The idea is to decompose the total manifold $M$ into JSJ pieces, where we are able to restrict taut foliations to these pieces, and finally discuss within graph manifolds. To avoid technical issues in Section~\ref{sec:5}, we will use the following weakened condition:

\begin{defn}
    Let $M$ be a compact, connected, orientable, and irreducible 3-manifold possibly with toral boundary. We say $M$ is \textit{separating} if
    \begin{enumerate}
        \item every embedded torus in $M$ is separating in $M$, and
        \item any connected graph or hyperbolic 3-submanifold $N\subset M$ consisting of JSJ pieces of $M$ can be embedded in a rational homology sphere.
    \end{enumerate}
    \label{defn:sep}
\end{defn}

It follows immediately that if $M$ is a rational homology sphere, then it is separating. Hence it suffices to prove the following proposition:

\begin{prop}
    Let $M$ be a closed separating 3-manifold, and $T\subset M$ an essential torus. Suppose $T$ cuts $M$ into two pieces $M_1,M_2$, such that $M=M_1\cup_{T}M_2$. Suppose $M$ admits a co-oriented taut foliation $\mathcal{F}$. Then there is a slope $[\alpha]$ of the torus $T$, such that it is CTF-detected as boundary slopes of both pieces $M_1$ and $M_2$.
    \label{prop:rhs}
\end{prop}

We will first make some observations and prepare some lemmas, and then proceed with the proof of Proposition~\ref{prop:rhs}.

\subsection{Observations and preparing lemmas} 

Since we want to find taut foliations in the submanifolds $M_1,M_2$, the most natural idea is to restrict the taut foliation $\mathcal{F}$ of $M$. However, such restrictions may not always be taut. Our first observation is that the only obstruction to tautness under restrictions is the Reeb annuli on the boundary:

\begin{lem}
    Let $M$ be a separating manifold, and $M'\subset M$ be an embedded 3-submanifold with toral boundary, such that components of $\partial M'$ are either components of $\partial M$ or lie in the interior of $M$. Suppose $M$ admits a co-oriented taut foliation $\mathcal{F}$ intersecting $\partial M$ transversely, and that $\mathcal{F}$ intersects $\partial M'-\partial M$ only in suspensions (i.e. no Reeb annulus). Then the restriction $\mathcal{F}|_{M'}$ is a co-oriented taut foliation of $M'$.
    \label{lem:reebless}
\end{lem}

\begin{proof}
    First suppose we have a leaf $F$ of $\mathcal{F}|_{M'}$ which intersects $\partial M'-\partial M$. Then by the Reebless supposition there are transverse closed circles intersecting $F$ on some components of $\partial M'-\partial M$. Now suppose we have a leaf $G$ of $\mathcal{F}|_{M'}$ not intersecting $\partial M'-\partial M$. Then it is also a leaf of $\mathcal{F}$, and we can find a transverse closed circle $\gamma$ in $M$ intersecting $G$. We parameterize $\gamma$ as $\gamma:[0,1]\rightarrow M$, such that $\gamma(0)=\gamma(1)\in G$, and that its orientation agrees with the co-orientation of $\mathcal{F}$. $\gamma$ may leave $M'$ through some component of $\partial M'-\partial M$. However, since $M$ is separating, each component of $\partial M'-\partial M$ is separating in $M$. Hence for each component $T$ of $\partial M'-\partial M$ intersecting $\gamma$, the preimage $\gamma^{-1}(T)$ has both a minimum $t_1$ and a maximum $t_2$, where $\gamma$ exits $\mathrm{int}(M')$ at $t_1$ and enters $\mathrm{int}(M')$ at $t_2$. Since $\mathcal{F}$ intersects $\partial M'-\partial M$ only in suspensions, we can replace the $\gamma$-arc between $t_1$ and $t_2$ by an arc on $T$ transverse to $\mathcal{F}$, complying with the co-orientation. After finitely many such replacements, we will get a transverse closed circle $\gamma'$ in $M'$ intersecting $G$.
\end{proof}

Our next observation is from~\cite{gabai1992taut}. Gabai has shown that if a knot manifold $M$ admits a taut foliation intersecting $\partial M$ transversely in a foliation with Reeb annuli, then any rational slope other than the slope represented by the Reeb annuli is strongly CTF-detected (\cite{gabai1992taut}, Corollary 3.1). When we only require the weaker condition of being CTF-detected, i.e. suspension foliations on the boundary, then actually all rational slopes are CTF-detected (we call such a knot manifold $M$ \textbf{persistently foliar}). Below we rewrite the conclusion in slope detection settings and give a slightly modified proof that works in our case:

\begin{lem}
    Let $M$ be a knot manifold. Suppose $M$ admits a co-oriented taut foliation $\mathcal{F}$ that intersects $\partial M$ transversely, while $\mathcal{F}|_{\partial M}$ contains some Reeb annuli. Then every rational slope of $\partial M$ is CTF-detected, i.e. $M$ is persistently foliar.
    \label{lem:pfoliar}
\end{lem}

\begin{proof}
    After possibly blowing up leaves intersecting $\partial M$ we can assume no two Reeb annuli share a boundary circle. Let $R_1, R_2,...,R_{2n}$ be the Reeb annuli of $\mathcal{F}|_{\partial M}$ in cyclic order, and $L_i,M_i$ be the two boundary circles of $R_i$ such that $L_1, M_1, L_2, M_2,...$ are in cyclic order. We can then further blow up the leaves of $\mathcal{F}$ corresponding to the circles $L_{2i-1},M_{2i-1}$ ($i=1,...,n$) to $I$-bundles. We still use $L_i,M_i$ ($i=1,...,2n$) to denote the boundary circles of $R_i$ after the blowing up, and further use $L_{2i-1}',M_{2i-1}'$ ($i=1,...,n$) to denote the other boundary circles of the annuli obtained by blowing up $L_{2i-1},M_{2i-1}$, such that locally $L_{2i-1}',L_{2i-1},M_{2i-1},M_{2i-1}'$ are in cyclic order, see the upper left of Figure~\ref{fig:pfoliar}. 

    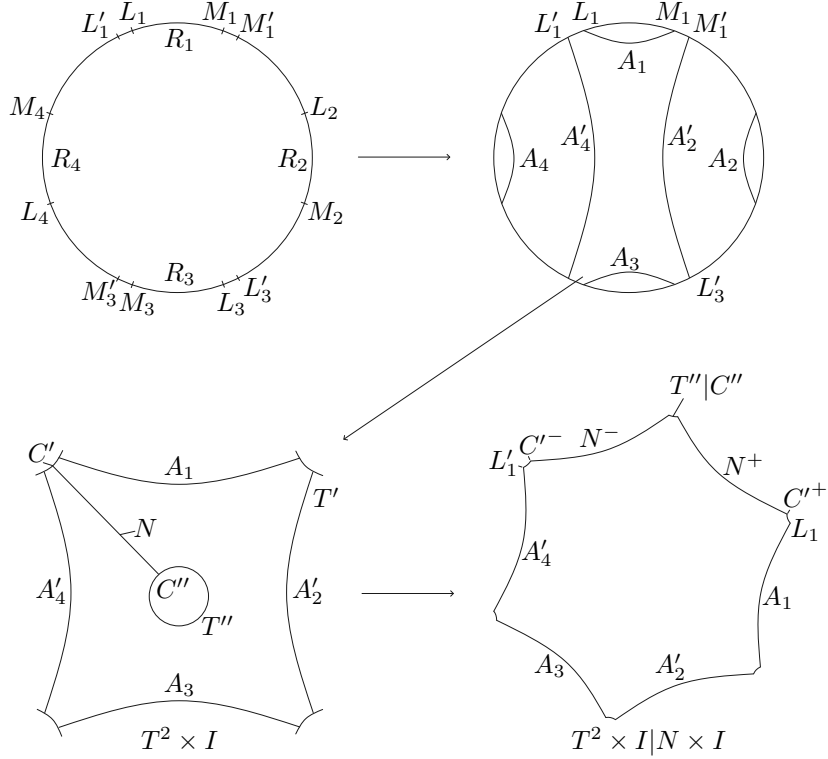
\begin{figure}[!hbt]
        \begin{overpic}[scale=0.4]{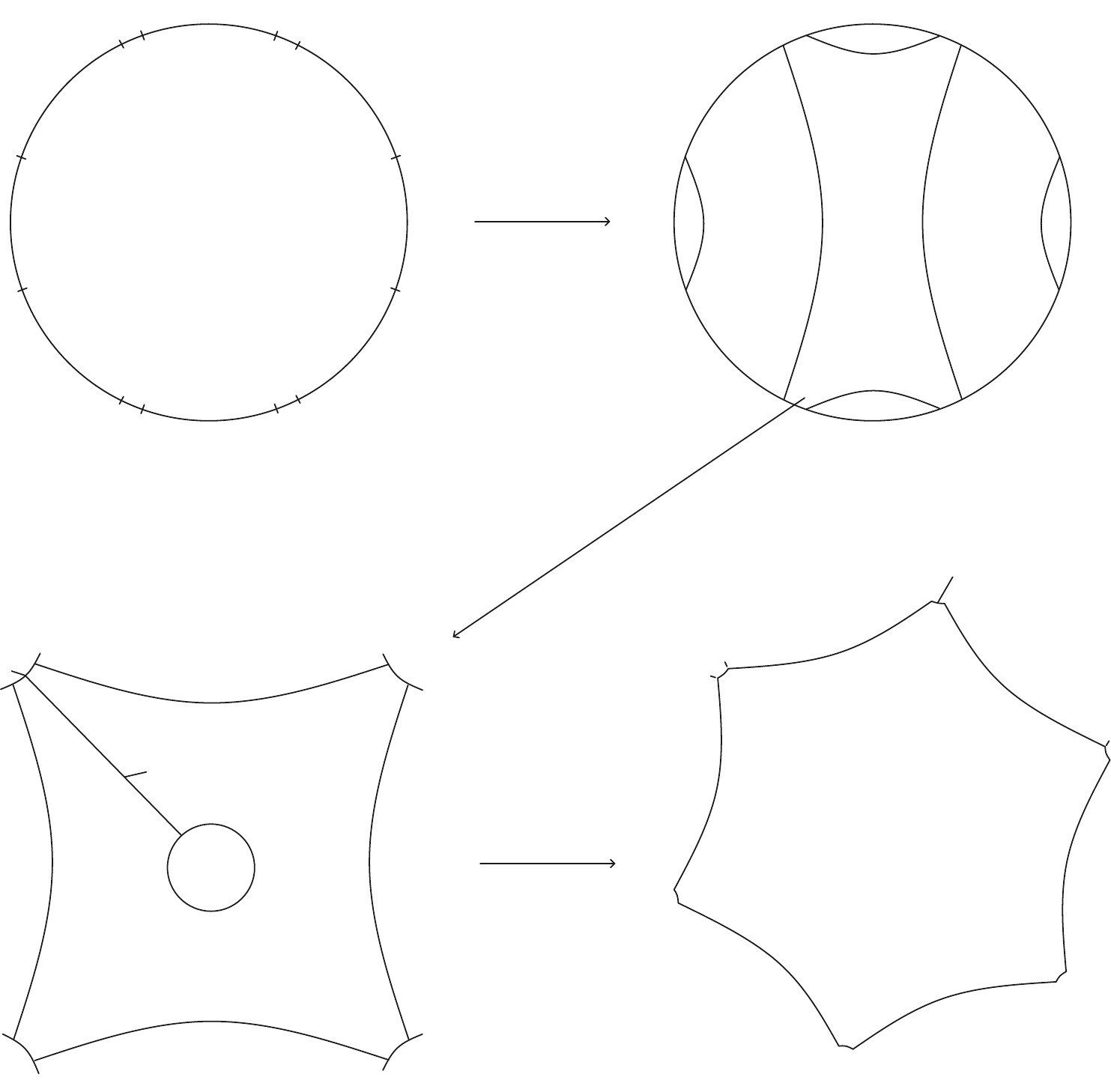}
            \put(17,93){$R_1$}
            \put(6,95){$L_1'$}
            \put(11,96.5){$L_1$}
            \put(22,96.5){$M_1$}
            \put(27,95){$M_1'$}
            \put(32,77){$R_2$}
            \put(36.3,84){$L_2$}
            \put(36,70){$M_2$}
            \put(17,62){$R_3$}
            \put(27.5,60){$L_3'$}
            \put(24,58){$L_3$}
            \put(11,58){$M_3$}
            \put(6,59.5){$M_3'$}
            \put(2,77){$R_4$}
            \put(-2,70){$L_4$}
            \put(-3.5,84){$M_4$}
            \put(77,90){$A_1$}
            \put(66,95){$L_1'$}
            \put(71,96.5){$L_1$}
            \put(82,96.5){$M_1$}
            \put(87,95){$M_1'$}
            \put(89,77){$A_2$}
            \put(83.6,80){$A_2'$}
            \put(76,64){$A_3$}
            \put(87.5,60){$L_3'$}
            \put(64,77){$A_4$}
            \put(69.5,80){$A_4'$}
            \put(17,36.3){$A_1$}
            \put(34,20){$A_2'$}
            \put(17,7.5){$A_3$}
            \put(0,20){$A_4'$}
            \put(13,28){$N$}
            \put(-1,38){$C'$}
            \put(16,20){$C''$}
            \put(22,15){$T''$}
            \put(36.8,32){$T'$}
            \put(14,0){$T^2\times I$}
            \put(71,0){$T^2\times I|N\times I$}
            \put(84,47){$T''|C''$}
            \put(91,36.3){$N^+$}
            \put(99,32){${C'}^+$}
            \put(100,28){$L_1$}
            \put(96.2,19){$A_1$}
            \put(82,10){$A_2'$}
            \put(66,10){$A_3$}
            \put(64.3,25){$A_4'$}
            \put(72,40){$N^-$}
            \put(64,39){${C'}^-$}
            \put(60,37){$L_1'$}
        \end{overpic}
        \caption{Persistent foliar decoration}
        \label{fig:pfoliar}
    \end{figure}
    
    Now we can perform the inverse of scalloping (see~\cite{gabai1992taut} Operation 2.3.1 or~\cite{calegari2007foliations} Example 4.20) to close up these Reeb annuli into annuli leaves that are boundary parallel to the original boundary torus. Let $A_i$ ($i=1,...,2n$) be the annulus leaf obtained by closing up $R_i$, so that $L_i$ and $M_i$ are connected by $A_i$ and lie in the same leaf in the new foliation.

    Now we further connect $M_{2i-1}'$ to $L_{2i+1}'$ by an annulus $A_{2i}'$ parallel to $A_{2i}$ for $i=1,...,n$. As a result, we obtain a new foliation $\mathcal{F'}$, with $n$ thickened annuli ``holes'' bounded by $A_{2i}$ and $A_{2i}'$ ($i=1,...,n$), and a new boundary torus $T'$ made up of $A_{2i-1},A_{2i}'$ and the blow-ups of $L_{2i-1}, M_{2i-1}$ ($i=1,...,n$), see Figure~\ref{fig:pfoliar} upper right.

    \cite{gabai1992taut} has already shown that any rational slope other than the slope represented by $L_1$ is strongly foliation-detected. In fact, we can glue an ideal polygon bundle over $S^1$ to $T'$ at the desired slope (i.e. the boundary of ideal polygon on $T'$ is of that slope), then by~\cite{calegari2007foliations} Example 4.22 we can foliate the ideal polygon bundle and extend to $\mathcal{F'}$. Then we can fill in the thickened annuli holes by~\cite{gabai1992taut} Operation 2.4.4. This gives a co-oriented taut foliation of the Dehn filling along the desired slope.

    We still need to prove the weak detection of the slope represented by $L_1$. The idea is similar. We can attach a thickened torus $T^2\times I$ to $T'$, and denote the other boundary torus $T''$. Now we can pick a circle $C'$ of $\mathcal{F'}|_{T'}$ in the interior of the blow up of $L_1$, and then an annulus $N=C\times I\subset T^2\times I$ bounded by $C'\subset T'$ and $C''\subset T''$. After cutting $T^2\times I$ along $N$, we can actually regard $T^2\times I|N$ as an ideal polygon bundle where the ideal polygon has $2n+2$ edges corresponding to the annuli $A_{2i-1},A_{2i}'$ ($i=1,...,n$) and the two sides of $N$, see bottom two pictures in Figure~\ref{fig:pfoliar}. We can then foliate this ideal polygon bundle and extend to $\mathcal{F'}$, and then fill in the thickened annuli holes. The resulting foliation is transverse to $T''$, and intersects $T''$ in a suspension foliation of the slope $[C'']$. Since $T''$ is parallel to the original boundary torus of $M$ and $C''$ is parallel to $L_1$, we got the slope detection desired.
\end{proof}

\begin{rem}
    The new foliation $\mathcal{F}'$ intersecting $\partial M$ in a suspension of desired slope can be obtained by blowing up leaves of the original foliation $\mathcal{F}$, filling in the $I$-product spaces with transverse leaves with possibly non-trivial holonomy, and then attaching a properly foliated $T^2\times I$ to the boundary. So for any essential torus $T^*$ in $M$, if $\mathcal{F}$ intersects $T^*$ in a suspension foliation, then $\mathcal{F}'$ would intersect $T^*$ in a suspension of the same slope, possibly bringing in new holonomies. This description is used in section~\ref{sec:5}.
    \label{rem:localpf}
\end{rem}

The next lemma gives a fundamental criterion for tautness under restriction and is well-known (see~\cite{boyer2021slope} Corollary 5.4).

\begin{lem}
    Let $M$ be a separating manifold. Suppose $M$ admits a co-oriented taut foliation $\mathcal{F}$ intersecting $\partial M$ transversely. Let $M_1\subset M$ be an embedded submanifold with boundary components that are either toral components of $\partial M$ or incompressible tori in the interior of $M$. If $\mathcal{F}|_{M_1}$ is not taut, then it contains an annulus leaf.
    \label{lem:essann}
\end{lem}

\begin{proof}
    For $\mathcal{F}|_{M_1}$ to be not taut there must be a dead end component. The boundary of this dead end component consists of leaves of $\mathcal{F}|_{M_1}$ and has Euler characteristic 0, thus must consist of tori and annuli. There cannot be any torus leaf since it will be separating while $\mathcal{F}$ is taut, so there must be some annulus leaf.
\end{proof}

Last but not least, to discuss slope detections in graph manifolds, we quote the following lemmas from~\cite{boyer2017foliations}:

\begin{lem}[\cite{boyer2017foliations}, Proposition 6.1(1)]
    Suppose $W$ is a Seifert manifold admitting a co-oriented taut foliation $\mathcal{F}$. If $W$ is a rational homology sphere, then $\mathcal{F}$ can be isotoped to be horizontal, and the base orbifold of $W$ has underlying space a 2-sphere.
    \label{lem:bcclosed}
\end{lem}

\begin{lem}[\cite{boyer2017foliations}, Proposition 6.6]
    Let $M$ be a compact orientable Seifert manifold that embeds in a rational homology sphere. Suppose a boundary multislope $[\alpha_{*}]$ is detected by some taut foliation $\mathcal{F}$. Then $[\alpha_{*}]$ is horizontal if and only if $\mathcal{F}$ is horizontal.
    \label{lem:bchori}
\end{lem}

\begin{lem}[\cite{boyer2017foliations}, Proposition 6.9]
    Let $M$ be a compact orientable Seifert manifold that embeds in a rational homology sphere. Let $[\alpha_{*}]$ be a boundary multislope of $M$, and $v([\alpha_{*}])$ the number of vertical slopes in $[\alpha_{*}]$.

    \begin{enumerate}[(1)]
        \item If $M$ has non-orientable base orbifold, then $[\alpha_{*}]$ is foliation detected if and only if $v([\alpha_{*}])\geq 1$.
        \item If $M$ has orientable base orbifold and $v([\alpha_{*}])>0$, then $[\alpha_{*}]$ is foliation detected if and only if $v([\alpha_{*}])\geq 2$.
    \end{enumerate}
    \label{lem:bcverti}
\end{lem}

\subsection{Main proof}

Now we are in position to prove Proposition~\ref{prop:rhs}. 

\begin{proof}[Proof of Proposition~\ref{prop:rhs}]
    We consider the JSJ decomposition of $M$. Since $M$ is separating, each of the decomposing torus is separating, and we can use a tree to represent the JSJ decomposition of $M$, where the nodes are the Seifert or hyperbolic pieces and the edges are the decomposing tori.

    Since $\mathcal{F}$ is taut, no decomposing torus can be isotoped to be a leaf of $\mathcal{F}$. Hence by~\cite{brittenham1999incompressible}, after possibly blowing up leaves of $\mathcal{F}$ we can isotope $\mathcal{F}$ such that it is transverse to the decomposing tori, and is an essential lamination when restricted to each JSJ piece.

    Now consider the position of our torus $T$. We consider the JSJ decomposing tori that are boundary components of some hyperbolic pieces. If $T$ can be isotoped to any of such decomposing tori, we then consider the foliation $\mathcal{F}|T$. If it is Reebless, by Lemma~\ref{lem:reebless} the slope of $\mathcal{F}|T$ is what we want. If $\mathcal{F}|T$ contains some Reeb annuli, we can consider the two pieces $M_1, M_2$, and suppose $M_1$ contains some hyperbolic piece whose boundary contains $T$. We then claim that $M_1$ is persistently foliar. In fact, $\mathcal{F}|_{M_1}$ cannot contain any annulus leaf, for otherwise this leaf would intersect the hyperbolic piece in essential annuli. Hence by Lemma~\ref{lem:essann} $\mathcal{F}|_{M_1}$ is taut. Now by Lemma~\ref{lem:pfoliar} $M_1$ is persistently foliar. We can then pick the longitudinal rational slope of $M_2$, which is CTF-detected by Gabai's sutured manifold decomposition~\cite{gabai1983foliations}. This rational slope would then be detected in both $M_1$ and $M_2$.
    
    Now suppose $T$ cannot be isotoped to any of the JSJ decomposing tori that are boundary components of some hyperbolic pieces. Then we claim that it can be isotoped to be disjoint from these decomposing tori. In fact, if it does essentially intersect such a decomposing torus, then it would also intersect some hyperbolic piece in essential annuli, leading to a contradiction.

    The JSJ decomposing tori that are boundary components of some hyperbolic pieces separate $M$ into pieces that are either hyperbolic or graph manifolds. We still can use a tree to describe such a decomposition, see Figure~\ref{fig:graphdecomp}, where our torus $T$ lies in some maximal graph manifold piece $G$ ($G$ could also be Seifert fibered).

    \begin{figure}[!hbt]
        \begin{overpic}[scale=0.5]{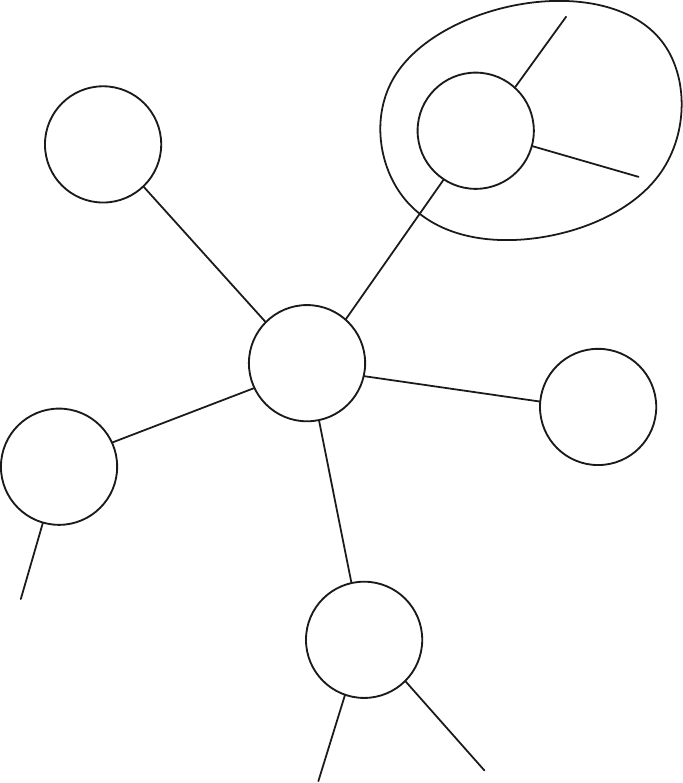}
            \put(37,51.5){$G$}
            \put(10,79){$H_1$}
            \put(58,81){$H_2$}
            \put(90,83){$\bar{H}_2$}
            \put(72.5,46){$H_3$}
            \put(43,16){$H_4$}
            \put(4,38){$H_5$}
        \end{overpic}
        \caption{Local picture around the graph manifold piece $G$}
        \label{fig:graphdecomp}
    \end{figure}
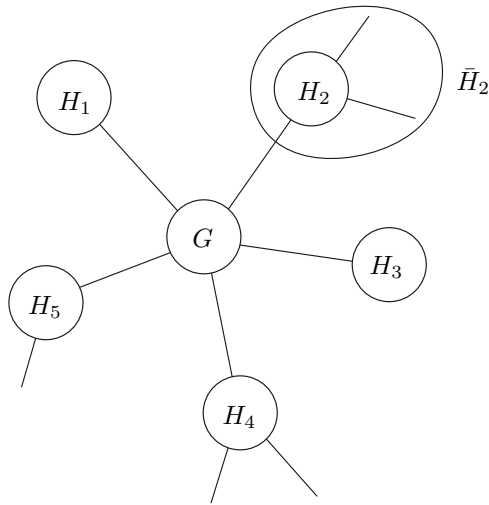

    Now consider the neighbours of $G$ on the new decomposing tree. Let $T_1,...,T_n$ be the boundary tori of $G$. Let $H_i$ be the hyperbolic piece whose boundary contains $T_i$, and $\bar{H}_i$ be the component of $M$ containing $H_i$ when separating $M$ along $T_i$, see Figure~\ref{fig:graphdecomp}. Again by Lemma~\ref{lem:essann}, $\mathcal{F}|_{\bar{H}_i}$ is taut.

    We can then consider how $\mathcal{F}$ intersects each $T_i$. Without loss of generality we may assume $\mathcal{F}$ intersects $T_1,...,T_k$ in foliations with Reeb annuli, and $T_{k+1},...,T_n$ in suspensions. Then by Lemma~\ref{lem:pfoliar}, $\bar{H}_{1},...,\bar{H}_k$ are persistently foliar. Let $G_1:=G\cup\bar{H}_{k+1}...\cup\bar{H}_n$ be the submanifold of $M$ obtained by cutting off the interiors of $\bar{H}_1,...,\bar{H}_k$. Now $(G_1,\partial G_1)$ is a taut sutured manifold with $H_2(G_1,\partial G_1)\neq 0$. By~\cite{gabai1983foliations}, there is then a taut foliation $\mathcal{F}_1$ of $G_1$ intersecting $\partial G_1$ transversely in suspensions of rational slopes.
    
    The new foliation $\mathcal{F}_1$ (after possibly blowing up some leaves) can be again isotoped to intersect $T_{k+1},...,T_n$ transversely with no boundary parallel annuli leaves in the pieces. We may again assume $\mathcal{F}_1$ intersects $T_{k+1},...,T_{k+l}$ in foliations with Reeb annuli and $T_{k+l+1},...,T_n$ in suspensions. Similarly to before we can claim that $\bar{H}_{k+1},...,\bar{H}_{k+l}$ are persistently foliar. We can then take $G_2:=G\cup\bar{H}_{k+l+1}...\cup\bar{H}_n$ and apply Gabai's theorem to $(G_2,\partial G_2)$, which gives a CTF-detected rational multislope $\{[\alpha_i]:\alpha_i\in T_i,\;1\leq i\leq k+l\}$ for $G_2$.

    We can repeat this procedure until we have $G_r:=G\cup \bar{H}_{m+1}...\cup \bar{H}_n$ (for some $m\leq n$), where $G_r$ admits a taut foliation $\mathcal{F}_r$ intersecting its boundary $T_1,...,T_m$ transversely in suspensions of rational slopes, and the essential tori $T_{m+1},...,T_n$ transversely in suspensions. Moreover, $\bar{H}_i$ is persistently foliar for $i\leq m$.

    Now by Lemma~\ref{lem:reebless}, $\mathcal{F}_r|_G$ is taut. We can then further consider the JSJ decomposition of $G$ into Seifert pieces. Suppose first that $T$ can be isotoped to some JSJ decomposing torus. It is then a classical argument by~\cite{brittenham1997graph} using the minimality of decomposing tori that $\mathcal{F}_r|_T$ must be Reebless. Since $\bar{H}_1,...,\bar{H}_m$ are persistently foliar, we can pick some taut foliation of $\bar{H}_i$ intersecting $T_i$ in the rational slope corresponding to $\mathcal{F}_r|_{T_i}$, and glue these $\bar{H}_i$ back to $G_r$. Then by the gluing theorem of~\cite{boyer2021slope} (see Remark~\ref{rem:multigluing}), we will eventually obtain a taut foliation $\mathcal{F}_r'$ of $M$ which is $\mathcal{F}_r$ when restricted to $G_r$ (up to blowing up leaves). Now that the taut foliation $\mathcal{F}_r'$ intersects $T$ in suspensions, by Lemma~\ref{lem:reebless} the slope of $\mathcal{F}_r'|_{T}$ is what we want.

    Now suppose $T$ is not isotopic to any JSJ decomposing torus. We claim again that $T$ can be isotoped to be disjoint from all JSJ tori. Suppose not. Then $T$ intersects 2 adjacent Seifert pieces $S_1,S_2$ in essential annuli $A_1,A_2$, such that $A_i\subset S_i$ ($i=1,2$) and $A_1,A_2$ share a boundary circle. However, each essential annulus in a Seifert manifold can be isotoped to be vertical, after possibly changing the Seifert fibration. We would then conclude that $S_1$ and $S_2$ share the same vertical slope on the boundary torus they intersect, contradicting the minimality of JSJ decomposition.

    It follows that, $T$ is eventually istoped into some Seifert manifold $S$. Then $T$ can be further isotoped to be either horizontal or vertical. If $T$ can be isotoped horizontal, then $S=M$ is closed; but then according to Lemma~\ref{lem:bcclosed}, the base orbifold of $S$ is orientable, and thus defines a co-orientation along the fibers, which in turn forces $T$ to be non-separating, leading to a contradiction. Hence $T$ can be isotoped vertical. Now $T$ further separates $S$ into two Seifert manifolds $S_1, S_2$. Suppose $S_1\subset M_1$.

    Now consider the foliation $\mathcal{F}_r|_S$. By Lemma~\ref{lem:reebless} again it is taut. If this foliation is horizontal, then it would intersect $T$ in suspensions, and as before by considering the glued-up taut foliation $\mathcal{F}_r'$ of $M$ we know the slope of $\mathcal{F}_r'|_T$ is detected in both $M_1$ and $M_2$. Now suppose $\mathcal{F}_r|_S$ is not horizontal. Then $S$ is necessarily not closed by Lemma~\ref{lem:bcclosed}, and by Lemma~\ref{lem:bchori} $\mathcal{F}_r|_S$ restricts to some boundary tori of $S$ in suspensions of vertical slopes.

    If the base orbifold of $S$ is orientable, then by Lemma~\ref{lem:bcverti} there are at least two vertical slopes in the boundary multislope of $S$ detected by $\mathcal{F}_r|_S$. If $S_1$ contains at least two boundary tori of $S$ with vertical slopes detected by $\mathcal{F}_r|_S$, then by Lemma~\ref{lem:bcverti} for any slope $[\alpha]$ of $T$ there is some taut foliation $\mathcal{F}_{\alpha}$ of $S_1$ intersecting $\partial S_1-T$ in the same slopes as $\mathcal{F}_r$, while intersecting $T$ in the given slope. Now by gluing $\mathcal{F}_r'|_{M_1-\mathring{S_1}}$ (taut by Lemma~\ref{lem:reebless}) to $\mathcal{F}_{\alpha}$ using the gluing theorem, we have actually shown that $M_1$ is persistently foliar. Picking the rational longitudinal slope of $M_2$ then closes this case. The same thing happens when $S_2$ contains at least two boundary tori of $S$ with vertical slopes. If both $S_1$ and $S_2$ contains exactly one vertical slope, then we can pick the slope of $T$ to be vertical, and by Lemma~\ref{lem:bcverti} we can get taut foliations $\mathcal{F}_{v_i}$ of $S_i$ ($i=1,2$), intersecting $\partial S_i-T$ in the same slopes as $\mathcal{F}_r|_S$, while intersecting $T$ in the vertical slope. By gluing $\mathcal{F}_r'|_{M-\mathring{S}}$ and these two foliations together, we get a taut foliation of $M$ intersecting $T$ in a suspension foliation, where we can then apply Lemma~\ref{lem:reebless}.

    Now suppose the base orbifold is not orientable. Then at least one of $S_1,S_2$ has a non-orientable base orbifold. Suppose $S_1$ has a non-orientable base orbifold. If $S_1$ also has a boundary torus of $S$ with vertical slope, then by Lemma~\ref{lem:bcverti} and a similar argument to above we can show that $M_1$ is actually persistently foliar. If $S_1$ does not contain any boundary torus of $S$ with vertical slope, then $S_2$ contains at least on such boundary torus of $S$, and by Lemma~\ref{lem:bcverti} again we can pick the slope on $T$ to be vertical to construct taut foliations for $S_1$ and $S_2$, and thus for $M_1$ and $M_2$.
\end{proof}

\section{Manifolds with positive first betti number}
\label{sec:3}

In this section we prove the following proposition:

\begin{prop}
    Let $M$ be a closed, orientable, and irreducible 3-manifold with $b_1(M)>0$. Let $T\subset M$ be a separating essential torus and $M=M_1\cup_T M_2$. Suppose $M$ admits a co-oriented taut foliation $\mathcal{F}$. Then there is a rational slope $[\alpha]$ of $T$ that is CTF-detected as boundary slopes of both $M_1$ and $M_2$.
    \label{prop:posbetti}
\end{prop}

The proof is just a slight modification to the method used in~\cite{gabai1983foliations}\cite{gabai1987foliations}. We recall the following definitions in~\cite{gabai1987foliations}:

\begin{defn}[\cite{gabai1987foliations}, Definition 1.3]
    An \textit{I-cobordism} between closed connected oriented surfaces $T_0$ and $T_1$ is a compact oriented 3-manifold $V$ such that $\partial V=T_0\cup T_1$ and for $i=0,1$ the induced maps $j_i: \mathrm{H}_1(T_i)\rightarrow \mathrm{H}_1(V)$ are injective.
\end{defn}

\begin{defn}[\cite{gabai1987foliations}, Definition 1.6]
    Let $M$ be a compact oriented 3-manifold, $S$ a properly embedded oriented surface in $M$, and $P$ a toral component of $\partial M$ such that $P\cap S=\varnothing$. $M$ is called \textit{$S_P$-atoroidal} if boundary parallel tori are the only surfaces $I$-cobordant to $P$ in $M-S$.
\end{defn}

\begin{proof}[Proof of Proposition~\ref{prop:posbetti}]
    We check the Mayer-Vietoris sequence with $\mathbb{Q}$-coefficients for $(M_1,M_2)$:

    $\xymatrix{
        ... \ar[r] &\mathrm{H}_1(T;\mathbb{Q}) \ar[r]^-{f} &\mathrm{H}_1(M_1;\mathbb{Q})\oplus \mathrm{H}_1(M_2;\mathbb{Q}) \ar[r] &\mathrm{H}_1(M;\mathbb{Q}) \ar[r]^-{0} &...
    }$
    ~\\

    If $\mathrm{ker} f$ is nonzero, then the rational slope corresponding to the kernel is longitudinal in both $M_1$ and $M_2$, and thus CTF-detected in both $M_1$ and $M_2$ by~\cite{gabai1983foliations}.

    Now suppose $\mathrm{ker} f=0$. Then one of $\mathrm{H_1}(M_i)$ ($i=1,2$) has rank $\geq 2$. Suppose this is $M_1$. Then we know $\mathrm{H_1}(M_1,\partial M_1)$ has rank $\geq 1$. Hence we can find a closed, orientable, non-separating, Thurston-norm-minimizing surface $S\subset M_1$ not intersecting $T$. We henceforth show that $M_1$ is actually persistently foliar, so that the longitudinal slope of $M_2$ (necessarily rational) is CTF-detected in both $M_1$ and $M_2$.

    If $M_1$ is $S_T-$atoroidal, then according to~\cite{gabai1987foliations} Theorem 1.7, the sutured manifold decomposition of $M_1$ using $S$ while avoiding $T$ ends up with sutured product pieces and a single $T^2\times I$, whose one boundary component is $T$ and the other is some $T'$ on which there are annuli sutures. We can actually regard $T'$ as the boundary of some ideal polygon bundle over $S_1$, where the $R_{\pm}$ part comes from the edges of the ideal polygon, and the annuli sutures are the cusps. By a similar argument to Lemma~\ref{lem:pfoliar}, we can foliate this $T^2\times I$, while observing the sutures, such that the foliation intersects $T$ in a suspension of any rational slope we want.

    If $M_1$ is not $S_T-$atoroidal, then again by~\cite{gabai1987foliations} there is some essential torus $T_1\subset M_1$ and some $I$-cobordism $V$ disjoint from $S$, such that $\partial V=T\cup (-T_1)$, and $N:=M_1-\mathring{V}$ is $S_{T_1}-$atoroidal. Then by our argument in the previous paragraph this $N$ is persistently foliar. We now consider $V$. Since it is an $I$-cobordism, for any rational slope $[\alpha]$ of $T$ (where $\alpha\in \mathrm{H}_1(T)$ represents the slope), there is a rational slope $[\beta]$ of $T_1$, such that $\alpha-\beta\in \mathrm{H}_1(\partial V)$ is null-homologous in $V$ (as the image of the homological map induced by inclusions). This implies the existence of a Thurston-norm-minimizing surface $S'$, representing a non-trivial class in $\mathrm{H}_2(V,\partial V)$, while intersecting $T$ and $T_1$ in slopes $[\alpha]$ and $[\beta]$ respectively. If we perform sutured manifold decomposition to $(V,\partial V)$ using $S'$, we would get a taut foliation intersecting $T$ and $T_1$ in suspensions of $[\alpha]$ and $[\beta]$ respectively. Moreover, since $N$ is persistently foliar, for any such $[\beta]$ we can (using Remark~\ref{rem:multigluing}) glue up $N$ and $V$ along $T_1$ and get a taut foliation of $M_1$ that intersects $T$ in a suspension of $[\alpha]$. Since the choice of $[\alpha]$ is arbitrary, we have actually shown that $M_1$ is persistently foliar.
\end{proof}

\section{Laminar branched surface and irrational detected slopes}
\label{sec:4}

In this section we show that irrational slopes are in the interior of the set of (CTF-)detected slopes in most cases. We will show that these are easy corollaries of the laminar branched surface theory.

\subsection{Laminar branched surfaces}

We first briefly review the basic definitions of branched surfaces. 

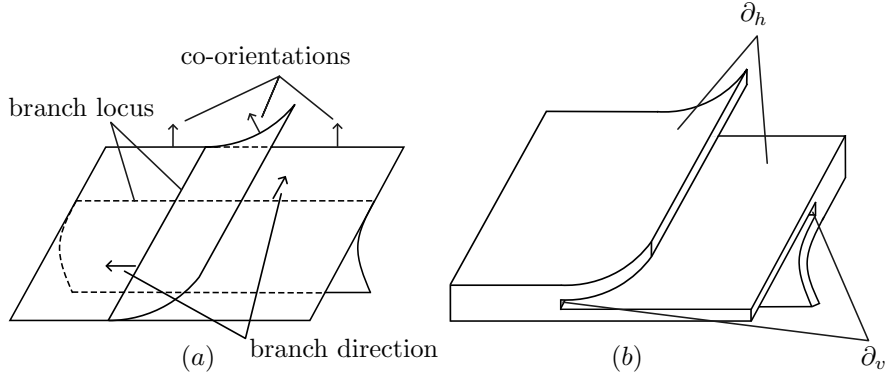
\begin{figure}[!hbt]
    \begin{overpic}[scale=0.6]{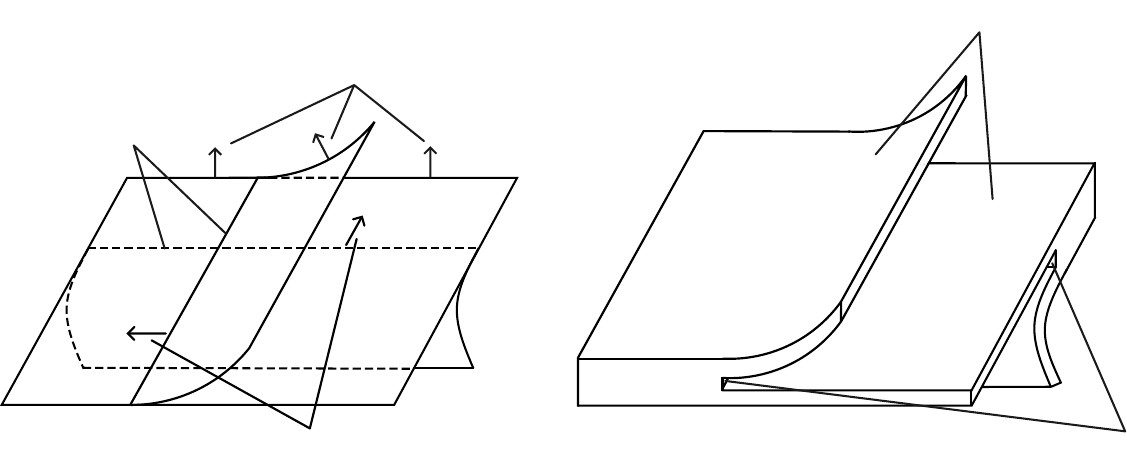}
        \put(0,29){branch locus}
        \put(20,35){co-orientations}
        \put(28,1){branch direction}
        \put(85,40){$\partial_h$}
        \put(99,0){$\partial_v$}
        \put(20,0){$(a)$}
        \put(70,0){$(b)$}
    \end{overpic}
    \caption{Branched Surfaces}
    \label{fig:brsf}
\end{figure}

A \textit{branched surface} $\mathcal{B}$ in a 3-manifold $M$ is a compact 2-complex locally modeled on Figure~\ref{fig:brsf}.$(a)$. The set $L$ of points not having a neighbourhood homeomorphic to $\mathbb{R}^2$ is called the \textit{branch locus} of $\mathcal{B}$. The components of $\mathcal{B}-L$ are called \textit{(branch) sectors} of $\mathcal{B}$. The normal direction of the branch locus pointing to the side with fewer sectors is called the \textit{branch direction}, see Figure~\ref{fig:brsf}.$(a)$. The branched surface is called \textit{co-oriented} if each branch sector is orientable and is equipped with a (co-)orientation, such that these (co-)orientations agree at the branch locus, as depicted in Figure~\ref{fig:brsf}.$(a)$.

By taking trivial $I$-bundles of each branch sector and pasting them together as in Figure~\ref{fig:brsf}.$(b)$ we can define $N(\mathcal{B})$ the \textit{regular neighbourhood} of $\mathcal{B}$. There is then a bundle projection $\pi:N(\mathcal{B})\rightarrow \mathcal{B}$ by collapsing the $I$-fibers. The part of $\partial N(\mathcal{B})$ consisting of the boundaries of $I$-fibers is called the \textit{horizontal boundary} of $N(\mathcal{B})$, denoted as $\partial_h N(\mathcal{B})$; the rest of $\partial N(\mathcal{B})$ consists of both interior segments of $I$-fibers at the branch locus of $\mathcal{B}$ and $I$-fibers transverse to the boundary of $\mathcal{B}$, and is called the \textit{vertical boundary} of $N(\mathcal{B})$, denoted as $\partial_v N(\mathcal{B})$, see Figure~\ref{fig:brsf}.$(b)$. A lamination $\mathcal{L}$ is \textit{carried} by $N(\mathcal{B})$ if it embeds in $N(\mathcal{B})$ while transverse to the $I$-fibers. It is \textit{fully carried} by $N(\mathcal{B})$ if moreover $\pi({\mathcal{L}})=\mathcal{B}$.

People were interested in which branched surfaces carry \textit{essential} laminations. This was answered by~\cite{gabai1989essential} using essential branched surfaces, and later improved by~\cite{li2002laminar} using laminar branched surfaces. In particular, in~\cite{li2002laminar} it was shown that a 3-manifold admits an essential lamination if and only if it admits a laminar branched surface. To introduce laminar branched surfaces here we first recall the following definition from~\cite{li2002laminar}:

\begin{defn}[\cite{li2002laminar}, Definition 1.3]
    Let $D^2\times I$ be a product component of $M-int(N(\mathcal{B}))$, where $D^2\times \partial I=D_1\sqcup D_2$ consists of disk horizontal boundary components of $N(\mathcal{B})$. Let $\pi:N(\mathcal{B})\rightarrow \mathcal{B}$ be the bundle projection. If $\pi$ restricted to the interior of $D_1\cup D_2$ is injective, i.e. each $I$-fiber of $N(\mathcal{B})$ intersects the interior of $D_1\cup D_2$ in either the empty set or a single point, then this $D^2\times I$ is called a \textit{trivial product region}, and $\pi(D_1\cup D_2)$ a \textit{trivial bubble} in $\mathcal{B}$.
    \label{defn:trivbub}
\end{defn}

Trivial product regions would destroy sink disks (to be introduced below) trivially, without changing the branched surface's ability of carrying laminations. In fact, if $K=D^2\times I$ is a trivial product region with respect to $\mathcal{B}$, then $N(\mathcal{B})\cup K$ is a regular neighbourhood of another branched surface $\mathcal{B}'$ obtained by collapsing the trivial bubble, and $N(\mathcal{B})$ carries a surface or lamination if and only if $N(\mathcal{B}')$ also carries it.

Now we are ready to define laminar branched surfaces.

\begin{defn}[\cite{li2002laminar}, Definition 1.4]
    A branched surface $\mathcal{B}$ in $M$ is called a \textit{laminar branched surface} if
    \begin{enumerate}
        \item $\partial_h N(\mathcal{B})$ is incompressible in $M-\mathrm{int}(N(\mathcal{B}))$, no component of $\partial_h N(\mathcal{B})$ is a sphere, and $M-B$ is irreducible;
        \item there is no monogon in $M-\mathrm{int}(N(\mathcal{B}))$, i.e. no disk $D$ properly embedded in $M-\mathrm{int}(N(\mathcal{B}))$ with $\partial D$ intersecting $\partial_v N(\mathcal{B})$ in a single interval fiber;
        \item there is no Reeb component, i.e. $\mathcal{B}$ does not carry a torus that bounds a solid torus; and
        \item $\mathcal{B}$ has no sink disk after collapsing all trivial bubbles. Here a \textit{sink disk} is a disk sector where along its boundary the branch direction always points inwards, see Figure~\ref{fig:skdsk}.$(a)$.
    \end{enumerate}
\end{defn}

\begin{figure}[!hbt]
    \begin{overpic}[scale=0.6]{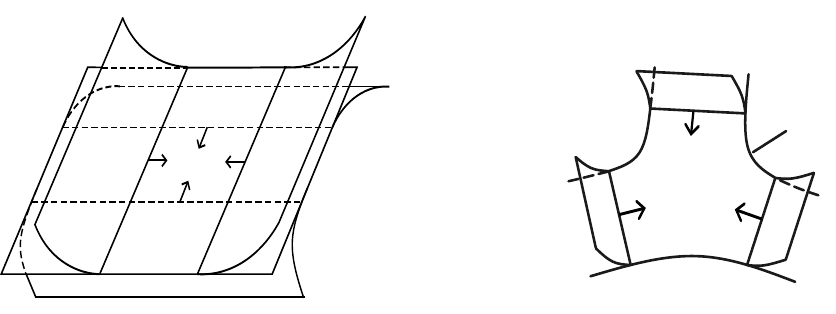}
        \put(18,-3){$(a)$}
        \put(81,-3){$(b)$}
        \put(95,22){$\partial \mathcal{B}$}
    \end{overpic}
    \caption{Sink disk and half sink disk}
    \label{fig:skdsk}
\end{figure}

When $M$ and $\mathcal{B}$ have nonempty boundary, we have a relative version of sink disks that gains importance. More precisely, a \textit{half sink disk} is a disk sector whose boundary arcs, if coming from the branch locus, always have branch directions pointing inwards, see Figure~\ref{fig:skdsk}.$(b)$.

Now we quote some modified results of~\cite{li2002laminar} and~\cite{li2003boundary} in the relative case:

\begin{prop}[modified from~\cite{li2002laminar}, Lemma 5.4]
    Let $M$ be a compact, connected, orientable, and irreducible 3-manifold with toral boundary. Let $\lambda$ be an essential lamination of $M$. Suppose $\lambda$ is not a lamination by simply connected leaves. Then there exists a laminar branched surface $\mathcal{B}$ carrying $\lambda$. Moreover, $\mathcal{B}$ has no half sink disk.
    \label{prop:safe}
\end{prop}

\begin{proof}
    The proof is essentially the same as that of~\cite{li2002laminar}, Lemma 5.4. In the relative case, we could take a triangulation $\mathcal{T}$ of $M$ so that $\partial M\subset \mathcal{T}^{(2)}$ and $\lambda$ is in normal form with respect to $\mathcal{T}$. We could still pick finitely many regular annuli and do necessary splittings for $\mathcal{B}$, so that eventually $N(\mathcal{B})\cap \mathcal{T}^{(2)}$ is contained in the safe region (as defined in~\cite{li2002laminar}). In particular, all branch sectors intersecting $\partial M$ are already contained in the safe region at this moment, and thus we can deal with the remaining 3-simplicies as if $M$ is closed and eventually get a laminar branched surface. By the definition of safe region, the branch sectors intersecting $\partial M$ either contain essential closed curves or have boundary arcs in the interior of $M$ with branch direction pointing outwards. Such sectors cannot be half sink disks.
\end{proof}

\begin{prop}[\cite{li2003boundary}, Theorem 2.2]
    Let $M$ be a knot manifold. Suppose $\mathcal{B}$ is a laminar branched surface without half sink disk and $\partial M-\partial \mathcal{B}$ is a union of bigons. For any rational slope $[r]\in \mathrm{PH}_1(\partial M;\mathbb{Q})$, let $M(r)$ be the Dehn filling along the rational slope $[r]$. If $[r]$ is realized by $\partial \mathcal{B}$, and $\mathcal{B}$ does not carry a torus that bounds a solid torus in $M(r)$, then $M(r)$ contains an essential lamination, which is carried by $\mathcal{B}$ when restricted to $M$.
    \label{prop:bdrytt}
\end{prop}

It remains for us to discuss what happens when there is an essential lamination by simply connected leaves in the relative case (in closed case this is~\cite{li2002laminar}, Prop. 4.2). In fact, the approach only essentially requires compactness, so we still have:

\begin{prop}[modified from~\cite{gabai1990foliations} or~\cite{li2002laminar}, Proposition 4.2]
    Let $M$ be a compact, connected, orientable, and irreducible 3-manifold containing an essential lamination $\lambda$ by simply connected leaves. Then $\pi_1(M)$ is commutative. In particular, if $M$ has non-empty boundary, then $M\cong S^1\times D^2$ or $M\cong T^2\times I$.
    \label{prop:planes}
\end{prop}

\subsection{Knot manifolds}

In this section we prove the following Proposition:

\begin{prop}
    Let $M$ be a separating knot manifold. Suppose there is a co-oriented taut foliation $\mathcal{F}$ intersecting $\partial M$ transversely in a suspension foliation of some irrational slope $[\alpha]$. Then there exists an open subset $U\subset \mathrm{PH}_1(\partial M;\mathbb{R})\cong S^1$, such that $[\alpha]\in U$, and any rational slope in $U$ is (strongly) CTF-detected.
    \label{prop:knotirr}
\end{prop}

\begin{proof}
    We blow up leaves of $\mathcal{F}$ to get a nowhere dense lamination $\mathcal{G}$. Since $M$ is a knot manifold, $\mathcal{F}$ is not a foliation by simply connected leaves. Then by Proposition~\ref{prop:safe} there is a laminar branched surface $\mathcal{B}$ without half sink disks that fully carries $\mathcal{G}$. $\mathcal{B}$ is necessarily co-oriented since $\mathcal{F}$ and $\mathcal{G}$ are. By blowing up and isotoping leaves, we may further suppose that $N_h(\mathcal{B})\subset \mathcal{G}$. Moreover, since $\mathcal{F}$ is a foliation, the complement of $N(\mathcal{B})$ consists of product regions. In order to use Proposition~\ref{prop:bdrytt} and for the new foliations to keep taut, we first prove the following claim:
    
    ~\\
    \textbf{Claim.} We can collapse trivial bubbles and do splittings on $\mathcal{B}$, so that it still satisfies the conclusion of Proposition~\ref{prop:safe}, and moreover carries no torus. 

    \begin{proof}[Proof of the claim.]
    
    If $\mathcal{B}$ carries some torus $T$, then since $M$ is separating, $T$ is necessarily separating. Since $\mathcal{B}$ is co-oriented, the bundle projection $N(\mathcal{B})\rightarrow \mathcal{B}$ must be injective when restricted to $T$, and we may think of $T$ as a subsurface of $\mathcal{B}$. 
    
    We claim that there is a leaf of $\mathcal{G}$ crossing $T$, i.e. carried by branch sectors on both sides of $T$. Suppose there are no such leaves. Then we can classify leaves of $\mathcal{G}$ into 3 classes: $\mathcal{T}$ containing those carried by $T$, $\mathcal{P}$ containing those carried by some sector on the positive side of $T$ (orientation of $T$ induced by the co-orientation of $\mathcal{F}$), and $\mathcal{N}$ containing those carried by some sector on the negative side of $T$. 
    
    If $\mathcal{T}$ is non-empty, then it is necessarily a sublamination of $\mathcal{G}$. Since it is carried by $T$ as a branched surface with no branching, the boundary leaf must be a torus, necessarily separating, contradicting the fact that $\mathcal{G}$ comes from a taut foliation.
    
    Now suppose $\mathcal{T}$ is empty. Then $\mathcal{P}$ and $\mathcal{N}$ are sublaminations, since no sequence of points in $\mathcal{P}$ can limit onto a sector on the negative side of $T$. Now $N(T)$ the regular neighbourhood of $T$ can be identified as a $T^2\times I$ region embedded in (the interior of) $N(\mathcal{B})$. We consider how $N(T)$ intersects $\mathcal{P}$ and $\mathcal{N}$. Since these are disjoint laminations, we may isotope $\mathcal{P}$ and $\mathcal{N}$ so that $N(T)\cap \mathcal{P}\subset T\times (\frac{1}{2},1]$ and $N(T)\cap \mathcal{N}\subset T\times [0,\frac{1}{2})$. It follows that there is a torus $T'=T\times \frac{1}{2}$ disjoint from $\mathcal{G}$. Since $\mathcal{G}$ comes from blowing up the foliation $\mathcal{F}$, $T'$ must lie in some complement product regions of $\mathcal{G}$. Since $T'$ is carried by $N(\mathcal{B})$, it is actually transverse to the $I$-fibers of the complement product region of $\mathcal{G}$, indicating that the product region is necessarily $T^2\times I$. But then $\mathcal{G}$ would contain a separating torus leaf, leading to a contradiction.

    It follows that there is some leaf crossing $T$. By splitting $\mathcal{B}$ along this leaf, we can make the branched surface to not carry $T$ any more. Since each torus carried by the original branched surface $\mathcal{B}$ is separating and thus injective under $I$-bundle projections, there are only finitely many classes of them (up to transverse isotopies in $N(\mathcal{B})$), each characterized by a set of branch sectors of $\mathcal{B}$. Since any surface carried by the branched surface after splittings is also carried by the original branched surface, after finitely many splittings we can make $\mathcal{B}$ carry no torus. We can then apply Proposition~\ref{prop:safe} again, where we collapse trivial bubbles and do necessary splittings to make $\mathcal{B}$ laminar without half sink disks. These operations do not create new torus carried by $\mathcal{B}$, so the resulting $\mathcal{B}$ still does not carry any torus.
    \end{proof}

    Since $[\alpha]$ is irrational, $\partial M-\partial \mathcal{B}$ necessarily consists of bigons. Now by Proposition~\ref{prop:bdrytt}, to any rational slope $[r]$ realized by $\partial \mathcal{B}$ there is an essential lamination $\mathcal{G}_r$ carried by $\mathcal{B}$ intersecting $\partial M$ in circles of slope $[r]$. Since the complement of $N(\mathcal{B})$ still consists of product regions, $\mathcal{G}_r$ extends to a foliation $\mathcal{F}_r$ intersecting $\partial M$ transversely in a product foliation of slope $[r]$. Since $\mathcal{B}$ does not carry any torus, $\mathcal{F}_r$ cannot contain any torus leaf. Moreover, to any possible annulus leaf of $\mathcal{F}_r$ there is a transverse closed circle crossing it on $\partial M$. It follows that $\mathcal{F}_r$ cannot have any dead end component, and is hence taut. Hence any rational slope $[r]$ carried by $\partial \mathcal{B}$ is strongly CTF-detected. It remains for us to show that $[\alpha]$ is in the interior of the set of slopes carried by $\partial \mathcal{B}$.

    We can assign a transverse measure to $\mathcal{F}|_{\partial M}$, which induces a transverse measure for $\mathcal{G}|_{\partial M}$, and eventually gives a transverse measure for $\partial \mathcal{B}$, i.e. a system of weights for $\partial \mathcal{B}$ (not necessarily positive). Let $\Omega(\partial \mathcal{B})$ be the space of all nonzero systems of weights of $\partial \mathcal{B}$, and $\Omega_{+}(\partial \mathcal{B})$ the space of positive systems of weights. If we take the classical linear space $V_{\partial \mathcal{B}}$ with a basis corresponding to the branches of $\partial \mathcal{B}$, then $\Omega(\partial \mathcal{B})\cup \{O\}$ is a closed cone of a linear subspace of $V_{\partial \mathcal{B}}$, and $\Omega_{+}(\partial \mathcal{B})$ is its interior. Now given a frame of $\partial M\cong T^2$ and $\gamma_0,\gamma_{\infty}$ the (0,1)- and (1,0)-curves in this frame, the slope associated to $\sigma$ a system of weights of $\partial \mathcal{B}$ can be interpreted as the ratio $[\sigma(\gamma_0):\sigma(\gamma_{\infty})]$, where $\sigma(\gamma_*)$ is the transverse measure of $\sigma$ valued along $\gamma_*$. Let $f:\Omega(\partial \mathcal{B})\rightarrow S^1$ be the map sending a system of weights to its associated slope. Then $f$ is the composition of a linear map $\Omega(\partial \mathcal{B})\rightarrow \mathbb{R}^2$ and the projection $\mathbb{R}^2\rightarrow \mathrm{P}\mathbb{R}^2\cong S^1$. Hence the image of $f$ is closed and connected, where the preimage of boundary points contains certain boundary edges of the cone $\Omega(\partial \mathcal{B})$. Since the linear equations and inequalities defining $\Omega(\partial \mathcal{B})$ are all rational, the images of these boundary edges under $f$ must be rational. It follows that the irrational slope $[\alpha]$ is an interior point of $f(\Omega(\partial \mathcal{B}))$, and thus an interior point of $f(\Omega_+(\partial \mathcal{B}))$. There is then an open set $U\subset S^1$, s.t. $[\alpha]\in U\subset f(\Omega_+(\partial \mathcal{B}))$. Now any (rational) point in $U$ would have a preimage in $\Omega_+(\partial \mathcal{B})$, indicating that this (rational) slope is fully carried (i.e. realized) by $\partial \mathcal{B}$.
\end{proof}

We are now ready to prove Theorem~\ref{thm:main}.

\begin{proof}[Proof of Theorem~\ref{thm:main}]
    We only need to prove the backward direction. Suppose $M$ admits a co-oriented taut foliation. If $M$ is not a rational homology sphere, then by Proposition~\ref{prop:posbetti}, we can find a rational slope CTF-detected on both sides. If $M$ is a rational homology sphere, then by Proposition~\ref{prop:rhs} we can find a slope CTF-detected on both sides. If this slopes is irrational, then by Proposition~\ref{prop:knotirr} we can find a nearby rational slope strongly CTF-detected on both sides.
\end{proof}

\subsection{General case}

In this subsection we give a generalisation to~\cite{boyer2017foliations}, Proposition 6.10. We first recall the following notion along with Proposition 6.10 in~\cite{boyer2017foliations}:

\begin{defn}
    Let $M$ be a compact, connected, orientable, and irreducible 3-manifold with boundary a disjoint union of incompressible tori $\partial M=T_1\sqcup...\sqcup T_n$. Let $J\subset \{1,...,n\}$. Let $[\alpha_*]=([\alpha_1],...,[\alpha_n])$, $[\alpha_i]\in \mathrm{PH}_1(T_i;\mathbb{R})$ be a boundary multislope of $M$. We say $(J;[\alpha_*])$ is CTF-detected if $M$ admits a co-oriented taut foliation $\mathcal{F}$, which intersects $T_i$ transversely in a suspension foliation of slope $[\alpha_i]$ for all $i$, and moreover intersects $T_i$ in a linear foliation of slope $[\alpha_i]$ for $i\in J$.
    \label{defn:jdet}
\end{defn}

\begin{prop}[\cite{boyer2017foliations}, Proposition 6.10]
    Let $M$ be a Seifert manifold with boundary that embeds in a rational homology sphere, $\partial M=T_1\sqcup...\sqcup T_n$. Let $[\alpha_*]$ be a boundary multislope of $M$. Suppose $J\subset \{1,...,n\}$ and $(J;[\alpha_*])$ is CTF-detected where some $[\alpha_j]$ is irrational. Suppose $j\in \{1,...,d\}$ if and only if $[\alpha_j]$ is irrational, and set $J^{\dagger}=J\cup \{1,...,d\}$. Then for $1\leq j\leq d$ there is an open set $U_j\subset \mathrm{PH}_1(T_j;\mathbb{R})$, such that one of the following statements holds:

    \begin{enumerate}
        \item $(J^{\dagger};[\alpha_*'])$ is CTF-detected for all rational multislope $[\alpha_*']$ such that $[\alpha_j']\in U_j$ for $j=1,...,d$ and $[\alpha_j']=[\alpha_j]$ otherwise.
        \item $M$ has no singular fibers, $d=2$, $J^{\dagger}=\{1,...,n\}$, and $[\alpha_*]$ is horizontal. Further, there is a homeomorphism $\varphi:U_1\rightarrow U_2$ which preserves both rational and irrational slopes and for which $(J^{\dagger};[\alpha_*'])$ is CTF-detected for all $[\alpha_*']=([\alpha_1'],\varphi([\alpha_1']),[\alpha_3],...,[\alpha_n])$ whenever $[\alpha_1']\in U_1$.
    \end{enumerate}
    \label{prop:graphirr}
\end{prop}

We will need the following multislope version of Proposition~\ref{prop:bdrytt}. Since the arguments in the original proof in~\cite{li2003boundary} using cusped product disks are local in nature, the original proof can be carried over here.

\begin{prop}[modified from~\cite{li2003boundary}, Theorem 2,2]
    Let $M$ be a compact, connected, orientable, and irreducible 3-manifold with boundary a disjoint union of incompressible tori $\partial M=T_1\sqcup...\sqcup T_n$. Suppose $\mathcal{B}$ is a laminar branched surface without half sink disk and for $i=1,...,d\leq n$, $T_i-((\partial \mathcal{B})\cap T_i)$ consists of bigons. Let $([\alpha_1],...,[\alpha_d])$ be any tuple of rational slopes realized by the boundary train tracks $(\partial \mathcal{B})\cap T_i$ ($i=1,...,d$), and $M_{\alpha,d}$ be the manifold (possibly still with boundary) obtained by doing Dehn fillings on $M$ along these slopes. If $\mathcal{B}$ does not carry a torus that bounds a solid torus in $M_{\alpha,d}$, then $M_{\alpha,d}$ contains an essential lamination, which is carried by $\mathcal{B}$ when restricted to $M$.
    \label{prop:mulbdrytt}
\end{prop}

Combining Proposition~\ref{prop:mulbdrytt} and Proposition~\ref{prop:planes}, we obtain the following partial generalisation to~\cite{boyer2017foliations}, Proposition 6.10:

\begin{prop}
    Let $M$ be a separating manifold. Suppose $\partial M=T_1\sqcup...\sqcup T_n$ is a disjoint union of tori, and there is some co-oriented taut foliation $\mathcal{F}$ detecting $(J,[\alpha_*])$, where $J\subset \{1,...,n\}$, and $[\alpha_*]=([\alpha_1],...,[\alpha_n])$ is a boundary multislope, $[\alpha_i]\in \mathrm{PH}_1(T_i,\mathbb{R})$. Suppose $d\geq 1$ and $j\in \{1,...,d\} \Leftrightarrow [\alpha_j]$ is irrational. Let $J^{\dagger}=J\cup \{1,...,d\}$. Then for $j=1,...,d$ there is an open set $U_j\subset \mathrm{PH}_1(T_j;\mathbb{R})$ such that one of the following two statements holds:

    \begin{enumerate}
        \item $(J^{\dagger};[\alpha_*'])$ is CTF-detected for all rational multislope $[\alpha_*']$ such that $[\alpha_j']\in U_j$ for $j=1,...,d$ and $[\alpha_j']=[\alpha_j]$ otherwise.
        \item $d=2$, $j\in J$ for all $j\geq 3$, and after doing Dehn fillings along $[\alpha_j]$ for all $j\geq 3$ we obtain $T^2\times I$.
    \end{enumerate}
    \label{prop:multirr}
\end{prop}

\begin{proof}
    We can perform Dehn fillings on all rational slopes strongly detected by $\mathcal{F}$ (suppose these slopes are $[\alpha_{l+1}],...,[\alpha_n]$ where $l\geq d$). $\mathcal{F}$ extends to a foliation $\mathcal{F}'$ on the resulting manifold $M'$. This $M'$ is not a solid torus since some irrational boundary (multi-)slope is CTF-detected. Suppose $M'\ncong T^2\times I$, then by Proposition~\ref{prop:planes} $\mathcal{F}'$ is not a foliation by planes and hence by Proposition~\ref{prop:safe} is carried by some laminar branched surface $\mathcal{B}$ without half sink disks. Moreover, performing similar operations as in the proof of Proposition~\ref{prop:knotirr} we can suppose that $\mathcal{B}$ carries no torus. Now by Proposition~\ref{prop:mulbdrytt} we know there exists open sets $U_j$ containing $[\alpha_j]$ for $j=1,...,d$, such that for any rational $[\alpha_j']\in U_j$ ($j=1,...,d$), $(\{1,...,d\};([\alpha_1'],...,[\alpha_d'],[\alpha_{d+1}],...,[\alpha_l]))$ is CTF-detected as a boundary multislope of $M'$. Since $\mathcal{F}'$ is transverse to the cores of the Dehn fillings we performed, the branched surface $\mathcal{B}$ is also transverse to these cores (up to isotopy). It follows that the new foliations detecting $[\alpha_j']$ are still transverse to these cores; hence by removing a small neighbourhood of these cores we recover the strong CTF-detections of $[\alpha_{l+1}],...,[\alpha_n]$. It then follows that $(J^{\dagger};[\alpha_*'])$ is CTF-detected.
\end{proof}

\section{Multislope detections}
\label{sec:5}

In this section we prove the following partial converse to~\cite{boyer2021slope}, Theorem 7.6.

\begin{prop}
    Let $M=\cup_j M_j$ be a closed separating manifold expressed as a union of submanifolds $M_1,...,M_{m+1}$ along disjoint essential tori $T_1,...,T_{m}$. Suppose $M$ admits a co-oriented taut foliation. Then there exists a rational multislope $([\alpha_1],...,[\alpha_m])\in \mathrm{PH}_1(T_1;\mathbb{R})\times...\times\mathrm{PH}_1(T_m;\mathbb{R})$ which is CTF-gluing coherent, that is, the corresponding boundary multislopes for $M_1,...,M_{m+1}$ are all CTF-detected.
    \label{prop:multidetect}
\end{prop}

\begin{proof}
    We do induction on $m$. When $m=1$ this is just Proposition~\ref{prop:rhs}. Now suppose the proposition holds for all $m<n$, for some $n\geq 1$.

    Consider a decomposition $M=\cup M_j$, where $M_1,...,M_{n+1}$ are glued along disjoint essential tori $T_1,...,T_n$. We may suppose that no two tori are parallel. Consider again the JSJ decomposition of $M$. Suppose there exists some maximal connected graph manifold piece $G$ containing at least one of the essential tori. Suppose after isotopies we have $T_1,...,T_k\subset G$ (some may be boundary parallel in $G$). Now we can apply the same strategy as in the proof of Proposition~\ref{prop:rhs} to get a co-oriented taut foliation $\mathcal{F}$ for $M$, where $\mathcal{F}|_{G}$ is a co-oriented taut foliation intersecting $\partial G$ in suspensions.

    We now consider the JSJ decomposing tori in $G$ along with all the boundary tori in $\partial G$. $\mathcal{F}$ intersects all these tori in suspensions, indicating that some multislope from these tori is CTF-gluing coherent. Some of the slopes may be irrational. However, by Proposition~\ref{prop:multirr} for each irrational slope we can replace it by some nearby rational slope, such that the new multislope is still CTF-gluing coherent (notice we don't have to deal with $T^2\times I$ if we don't require strong CTF-detections). Let $\mathcal{F}'$ be a taut foliation of $M$ detecting this rational multislope.

    If any of $T_1,...,T_k\subset G$ is parallel to some components of $\partial G$ or some JSJ decomposing tori of $G$, then $\mathcal{F}'$ already intersects them in suspensions of rational slopes. We now consider the rest of $T_1,...,T_k$ that are essential tori in some Seifert pieces $S_1,...,S_h$. Again similar to the proof of Proposition~\ref{prop:rhs}, by applying Lemmas~\ref{lem:bcclosed} - \ref{lem:bcverti} with an induction on the number of essential tori, we can modify the restrictions $\mathcal{F}'|_{S_j}$ to intersect these essential tori in suspensions of rational slopes, without changing the (rational) slopes of $\mathcal{F}'|_{\partial S_j}$. Gluing these new foliations on $\cup S_j$ back to $\mathcal{F}'|_{M-\cup S_j}$, we obtain a taut foliation $\mathcal{F}''$ of $M$, intersecting $T_1,...,T_k$ and $\partial G$ in suspensions of rational slopes.

    Let $H_1,...,H_d$ be the hyperbolic pieces in the JSJ decomposition of $M$ adjacent to $G$, and $\bar{H}_i$ be the component of $M-G$ containing $H_i$, as defined in the proof of Proposition~\ref{prop:rhs} (see Figure~\ref{fig:graphdecomp}). Suppose some $\bar{H}_i$ contains $T_{k+1},...,T_{k+l}$. Now $\mathcal{F}''|_{\bar{H}_i}$ is a taut foliation intersecting the boundary torus in a suspension of some rational slope $[\beta_i]$. We can then glue $N_2$ a twisted $I$-bundle over Klein bottle to $\bar{H}_i$, identifying $[\lambda]$ the longitude of $N_2$ with $[\beta_i]$. Denote the resulting manifold $\tilde{H}_i$. $\tilde{H}_i$ admits a taut foliation by gluing together $\mathcal{F}''|_{\bar{H}_i}$ and the horizontal foliation of $N_2$ by annuli. We check that $\tilde{H}_i$ is still separating. In fact, $\tilde{H}_i$ is obtained by gluing two knot manifolds together, hence it is still irreducible, and any compressing torus is separating. Moreover, any essential torus in $\tilde{H}_i$ must avoid $H_i$ (which is hyperbolic), hence embeds in $\bar{H}_i$ and is separating. The characteristic submanifold of $\tilde{H}_i$ would be the disjoint union of the characteristic submanifold of $\bar{H}_i$ and $N_2$, so any connected graph or hyperbolic submanifold consisting of JSJ pieces still embeds in a rational homology sphere. It follows that we can apply our induction hypothesis to get some taut foliation $\mathcal{H}_i$ of $\tilde{H}_i$ intersecting $T_{k+1},...,T_{k+l}$ in suspension foliations of rational slopes.

    Now consider $\mathcal{H}_i|_{\bar{H}_i}$. If it intersects $\partial \bar{H}_i$ in a suspension, then it must intersect it in a suspension of slope $[\beta_i]$, since $\mathcal{H}_i|_{N_2}$ would then be taut, while $[\lambda]$ is the only CTF-detected slope for $N_2$. Suppose it intersects $\partial \bar{H}_i$ in a foliation with Reeb annuli. Then since $H_i$ is hyperbolic, $\mathcal{H}_i|_{\bar{H}_i}$ cannot have annuli leaves, and is hence taut. Now by Lemma~\ref{lem:pfoliar} $\bar{H}_i$ is persistently foliar, and we can get a taut foliation $\mathcal{H}_i'$ intersecting $\partial \bar{H}_i$ in a suspension foliation of slope $[\beta_i]$. Moreover, by Remark~\ref{rem:localpf}, this $\mathcal{H}_i'$ would still intersect $T_{k+1},...,T_{k+l}$ in suspensions of rational slopes.

    We can do such operations for all $\bar{H}_i$ (if $\bar{H_i}$ does not contain any $T_j$ we just let $\mathcal{H}_i'=\mathcal{F}''|_{\bar{H}_i}$). By gluing together the resulting $\mathcal{H}_i'$ and $\mathcal{F}''|_G$, we obtain a foliation $\mathcal{G}$ detecting some rational multislope on $T_1,...,T_n$, as desired.

    If we cannot find any nontrivial graph manifold $G$ from the JSJ decomposition containing some of $T_1,...,T_n$, then each torus is a JSJ-decomposing torus, and they are all connecting hyperbolic pieces. This case is essentially the same. In fact, we can apply Proposition~\ref{prop:rhs} and Proposition~\ref{prop:knotirr} to get a foliation $\mathcal{F}$ intersecting $T_1$ in a suspension foliation of some rational slope $[r]$. We can then cut $M$ along $T_1$, attach $N_2$ to the two pieces respectively with longitude $[\lambda]$ attached to $[r]$, and use induction hypothesis on the two resulting closed manifolds. Eventually we still obtain a foliation $\mathcal{G}$ detecting some rational multislope on $T_1,...,T_n$.
\end{proof}

\bibliography{references.bib}
\bibliographystyle{alpha}

\end{document}